\documentclass[11pt,a4paper,oneside,reqno]{emspublic}
\usepackage[english]{babel}
\usepackage{amsopn}

\theoremstyle{definition}

\usepackage[square,sort,comma,numbers]{natbib}

\usepackage[all]{xy}
\usepackage{pstricks}
\usepackage{enumerate}
\usepackage{amsfonts,amssymb,amsmath,eucal,pinlabel,array,hhline}
\usepackage{slashed}
\usepackage{tabulary}
\usepackage{fancyhdr}
\usepackage{calrsfs,bbm,dsfont}
\usepackage[position=b]{subcaption}

\newcommand{\id}{\mathit{id}}
 
\begin{document}



\title{Burau maps and twisted Alexander polynomials}
\author{Anthony Conway}
\affiliation{Section de Math\'ematiques, Universit\'e de Gen\`eve, 2-4 rue du Li\`evre, 1227 Acacias, Geneva, Switzerland}
\maketitle

\begin{abstract}
The Burau representation of the braid group can be used to recover the Alexander polynomial of the closure of a braid. We define twisted Burau maps and use them to compute twisted Alexander polynomials.
\end{abstract}

\amsprimary{57M25} 
\section{Introduction}
\label{sec:intro}

As any link can be obtained as the closure of a braid, one can hope to recover link invariants from the braid groups. This idea was first exploited by Burau \citep{Burau} in order to compute the Alexander polynomial. More precisely, if  
$$ \overline{\mathcal{B}}_t : B_n \longrightarrow GL_{n-1}(\mathbb{Z}[t^{\pm 1}])~$$
denotes the reduced Burau representation of the braid group and~$\beta \in B_n$ is a braid, then Burau showed that 
$$ \Delta_{\hat{\beta}}(t)(t^n-1)=\pm t^m\det(\overline{\mathcal{B}}_{t}(\beta)-I_{n-1})(t-1),$$
where $m \in \mathbb{Z}$, ~$ \Delta_{\hat{\beta}}(t)$ denotes the Alexander polynomial of the link~$\hat{\beta}$ obtained by closing~$\beta$, and~$I_k$ is the identity~$(k\times k)$ matrix. Some years later, Birman \citep{Birman} generalized this result to compute the multivariable Alexander polynomial~$\Delta_{\hat{\beta}}(t_1,...,t_n)$ of the closure of a pure braid~$\beta$. Indeed using the reduced Gassner representation
$$ \overline{\mathcal{B}}_{t_1,\dots,t_n}: P_n \rightarrow GL_{n-1}(\mathbb{Z}[t_1^{\pm 1},\dots,t_n^{\pm 1}])$$
of the pure braid group~$P_n$, Birman showed that
$$ \Delta_{\hat{\beta}}(t_1,...,t_n)(t_1t_2\cdots t_n-1)=\pm t_1^{m_1}t_2^{m_2} \cdots t_n^{m_n} \text{det}(\overline{\mathcal{B}}_{t_1,\dots,t_n}(\beta)-I_{n-1}),$$
for $m_i \in \mathbb{Z}$. A braid~$\beta$ is{\em~$\mu$-colored\/} if each of its $n$ components is assigned an element in~$\{1,2,\dots,\mu\}$, resulting in a sequence~$c=(c_1,c_2,\dots,c_n)$ of integers. If one fixes such a sequence~$c=(c_1,c_2,\dots,c_n)$, one obtains the \textit{colored braid group}~$B_c$ (see Subsection (\ref{sub:braids}) for the precise definition). As $c$ varies, the group $B_c$ interpolates between the braid group~$B_n$ (when~$c=(1,1,\dots,1)$) and the pure braid group~$P_n$ (when~$\mu=n$ and~$c=(1,2,\dots,n)$). In particular, if~$\beta$ is a~$\mu$-colored braid and~$\tau(\hat{\beta})$ denotes the torsion of the $\mu$-colored link~$\hat{\beta}$, then the results of Burau and Birman can be written in a single equation as
$$  \tau(\hat{\beta})(t_1, \dots, t_\mu)(t_{c_1}t_{c_2}\cdots t_{c_n}-1)=\pm t_1^{m_1}t_2^{m_2} \cdots t_\mu^{m_\mu}\det(\overline{\mathcal{B}}_{t_1,\dots,t_\mu}(\beta)-I_{n-1}),$$
with $m_i \in \mathbb{Z}.$ The Alexander polynomial was further generalized (independently by Jiang-Wang \citep{Jiang-Wang} and Lin \citep{Lin}) to the so-called \textit{twisted Alexander polynomial}. This polynomial, denoted~$\Delta_L^\rho$, depends both on a link~$L$ and on a representation~$\rho$ of the group~$ \pi_1(S^3 \setminus L)$. Since then, the twisted Alexander polynomial has proved wildly successful in knot theory and low-dimensional topology. Indeed its applications range from periodic knots \cite{HLN} to knot concordance \cite{KL,KL2,Friedl-Cha}, while also giving lower bounds on the Thurston norm and providing obstructions to fiberedness \cite{Friedl-Kim,FK2} (see \cite{survey} for a survey).

As Kitano  \citep{Kitano} showed that~$\Delta_L^\rho$ can also be interpreted via a twisted torsion~$\tau^\rho(L)$, it is natural to wonder whether the twisted Alexander polynomial may be recovered from braids. The main aim of this paper is to answer this question. 

Given a commutative ring $R$ and a representation~$\rho: F_n \rightarrow GL_k(R)$ of the free group, we will define a reduced twisted Burau map 
$$ \overline{\mathcal{B}}_\rho\colon B_c \rightarrow GL_{(n-1)k}(R[t_1^{\pm 1},\dots, t_\mu^{\pm 1}])$$
which is a twisted analogue of the reduced Burau representation (see Subsection (\ref{sub:reduced}) for a precise definition). Given a colored braid~$\beta \in B_c$, our main theorem relates~$\overline{\mathcal{B}}_\rho(\beta)$ to the twisted torsion~$\tau^\rho({\hat{\beta}})$.

\begin{theorem}
Let~$F_n$ be the free group on~$x_1, x_2,\dots, x_n$ and let~$\beta \in B_c$ be a~$\mu$-colored braid  with~$n$ strands. If~$\rho\colon F_n \rightarrow GL_k(R)$ is a representation which extends to~$\pi_1(S^3 \setminus \hat{\beta})$, then
\[
\tau^\rho(\hat{\beta})(t_1,\dots,t_\mu ) \det\left(\rho(x_1 x_2 \cdots x_n) t_{c_1}t_{c_2}\cdots t_{c_n}-I_{k}\right)= \pm d t_1^{m_1}t_2^{m_2} \cdots t_\mu^{m_\mu} \det(\overline{\mathcal{B}}_\rho(\beta)-I_{(n-1)k}),
 \]
for some $d \in \det(\rho(\pi_1(S^3 \setminus \hat{\beta})))$ and $m_i \in \mathbb{Z}.$
\end{theorem}
 
The second aim of this article is to study the properties of the twisted Burau maps. In the classical case, the Burau representation can be defined via Fox calculus or by using the homology of covering spaces.  In Subsections (\ref{sub:Burau}) and (\ref{sub:reduced}), we shall investigate to what extent these constructions can be generalized to the twisted case.

The paper is organized as follows. In Section~$\ref{sec:preliminaries}$, we recall the necessary definitions: colored braids, twisted homology, twisted intersection forms, the torsion of a chain complex and the twisted torsion of links. In Section \ref{sec:results}, we introduce the twisted Burau maps (Subsection (\ref{sub:Burau})), the reduced twisted Burau maps (Subsection (\ref{sub:reduced})) and  prove the main result (Subsection (\ref{sub:thm})).

\subsection*{Acknowledgments.} The author wishes to thank warmly his advisor David Cimasoni and Stefan Friedl, as well as Hans Boden for useful conversations.  This work was supported by the NCCR SwissMap, funded by the Swiss FNS.

\section{Preliminaries}
\label{sec:preliminaries}

\subsection{Colored braids}
\label{sub:braids}
Following Birman \citep{Birman}, we start by recalling some well-known properties of the braid group. Afterwards, slightly modifying some conventions of \citep{CT,CC}, we discuss colored braids.

Let $D^2$ be the closed unit disk in $\mathbb{R}^2.$ Fix a set of $n \geq 1$ punctures $p_1,p_2,\dots,p_n$ in the interior of $D^2$. We shall assume that the $p_i$ lie in $(-1,1)=Int(D^2) \cap \mathbb{R}$ and $p_1<p_2<\dots<p_n.$ A \textit{braid with $n$ strands} is an oriented $n$-component one-dimensional submanifold $\beta$ of the cylinder $D^2 \times [0,1]$ whose boundary is $\bigsqcup_{i=1}^n (p_i \times \lbrace 0) \rbrace \sqcup \left(-\bigsqcup_{i=1}^n (p_i \times \lbrace 1 \rbrace) \right)$, and where the projection to $[0,1]$ maps each component of $\beta$ homeomorphically onto $[0,1]$. Two braids $\beta_1$ and $\beta_2$ are \textit{isotopic} if there is a self-homeomorphism of $D^2 \times [0,1]$ which keeps $D^2 \times \lbrace 0,1 \rbrace$ fixed, such that $h(\beta_1)=\beta_2$ and $h|_{\beta_1}:\beta_1 \simeq \beta_2$ is orientation preserving. The \textit{braid group} $B_n$ consists of the set of isotopy classes of braids.  The identity element is given by the \textit{trivial braid} $\lbrace p_1,p_2,\dots, p_n \rbrace \times [0,1]$ while the composition $\beta_1 \beta_2$ consists in gluing $\beta_1$ on top of $\beta_2$ and shrinking the result by a factor $2$ (see Figure \ref{fig:CompositionTwistedOriente}). 

\begin{figure}[h]
\labellist\small\hair 2.5pt
\pinlabel {$x_1$} at 35 750
\pinlabel {$x_2$} at 80 750
\pinlabel {$x_3$} at 125 750
\pinlabel {$z$} at 80 839
\endlabellist
\centering
\includegraphics[width=0.4\textwidth,scale=0.6]{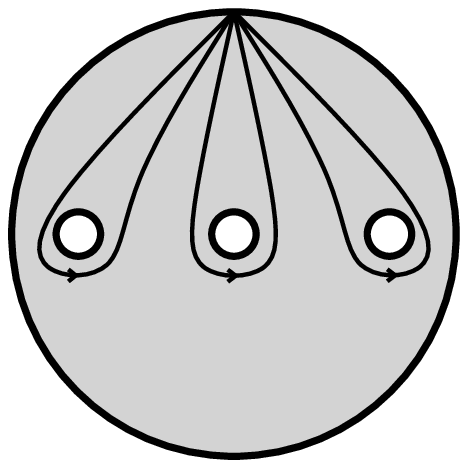}
\caption{The punctured disk $D_3$. }
\label{fig:DiskTwisted}
\end{figure}

The braid group~$B_n$ can also be seen as the set of  isotopy classes of orientation-preserving homeomorphisms of~$D_n :=D^2 \setminus \lbrace p_1,\dots, p_n \rbrace$ fixing the boundary pointwise. Either way, $B_n$ admits a presentation with $n-1$ generators $\sigma_1,\sigma_2, \dots, \sigma_{n-1}$ subject to the relations $\sigma_i \sigma_{i+1} \sigma_i=\sigma_{i+1} \sigma_i \sigma_{i+1}$ for each $i$, and $\sigma_i \sigma_j = \sigma_j \sigma_i$ if $|i-j|>2$. Topologically, the generator $\sigma_i$ is the braid whose $i$-th component passes over the $i+1$-th component. Sending a braid to its underlying permutation produces a surjection from the braid group into the symmetric group. The kernel $P_n$ of this map is called the \textit{pure braid group}. 

 \begin{figure}[h]
\labellist\small\hair 2.5pt
\pinlabel {$\beta$} at 60 720
\pinlabel {$\hat{\beta}$} at 350 720
\endlabellist
\centering
\includegraphics[width=0.4\textwidth,scale=1.2]{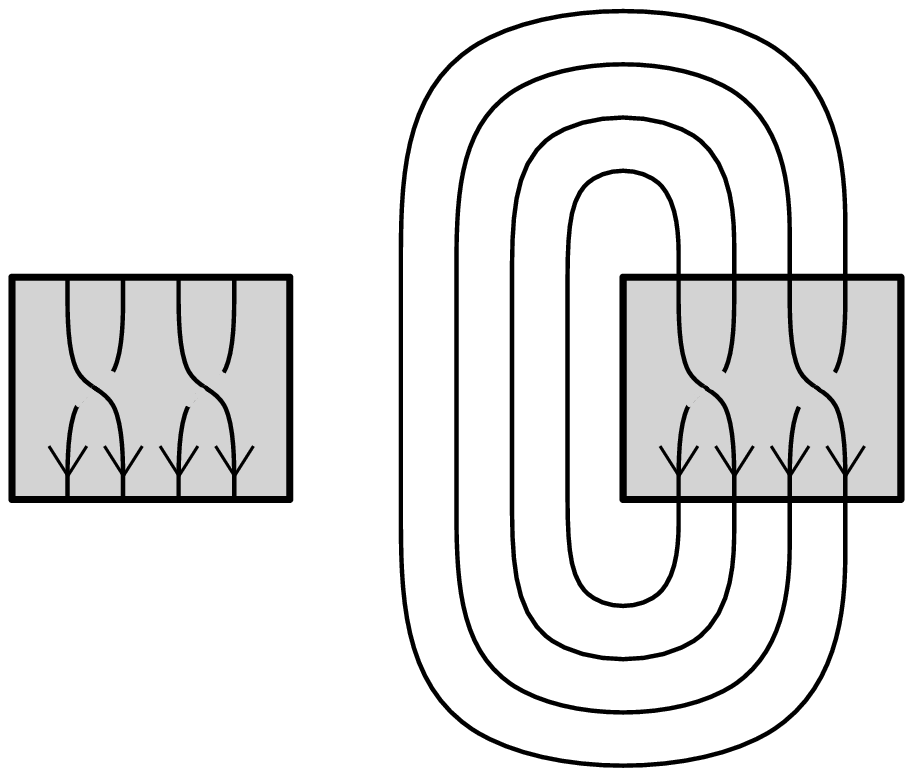}
\caption{The closure of a braid.}
\label{fig:ClosureTwistedOriente}
\end{figure}

Fix a base point $z$ of $D_n$ and denote by $x_i$ the simple loop based at $z$ turning once around $p_i$ counterclockwise for $i=1,2,\dots, n$ (see Figure \ref{fig:DiskTwisted}). The group $\pi_1(D_n)$ can then be identified with the free group $F_n$ on the $x_i.$ 
If $h_\beta$ is a homeomorphism of $D_n$ representing a braid $\beta$, then the induced automorphism $h_{\beta*}$ of the free group $F_n$ only depends on $\beta$. It follows from the way we compose braids that $h_{(\gamma \beta)*}=h_{\beta*}h_{\gamma*}$, and the resulting \textit{right} action of $B_n$ on $F_n$ can be explicitly described by
\[
x_j\sigma_i=
\begin{cases}
x_i x_{i+1} x_i^{-1}  & \mbox{if }  j=i, \\
 x_i                            & \mbox{if }  j=i+1, \\ 
x_j                             & \mbox{otherwise. } \\ 
 \end{cases} 
\] 
For this reason, $\theta_1 \theta_2$ denotes the \textit{left to right} composition of  $\theta_1,\theta_2 \in Aut(F_n)$ i.e. if $x \in F_n$, then $(x)\theta_1 \theta_2=((x)\theta_1) \theta_2.$ Moreover, if $f: F_n \rightarrow G$ is a group homomorphism, then $\beta_*f$ will denote the composition of $f$ with the automorphism induced by $\beta$. With these conventions, if $\beta$ and $\gamma$ are two braids, then $(\beta \gamma)_*f=\beta_*\gamma_*f.$

The \textit{closure} of a braid $\beta$ is the link $\hat{\beta}$ obtained from $\beta$ by adding parallel strands in $S^3 \setminus (D^2 \times [0,1])$ (see Figure \ref{fig:ClosureTwistedOriente}). If~$\beta$ is a braid with $n$ strands, then~$\pi_1(S^3 \setminus \hat{\beta})$ admits a presentation where the $n$ generators $x_1, x_2,\dots,x_n$ are subject to the relations $x_i=x_i \beta $ for $ i=1,2,\dots,n.$ In particular any group homomorphism $f: F_n \rightarrow G$ satisfying $\beta_*f=f$ extends to~$\pi_1(S^3 \setminus \hat{\beta})$. 

\begin{figure}[h]
\labellist\small\hair 2.5pt
\pinlabel {$\beta_1$} at 1 760
\pinlabel {$\beta_2$} at 1 615
\pinlabel {$\beta_1 \beta_2$} at 390 685
\pinlabel {$c$} at 70 809
\pinlabel {$c'$} at 70 708
\pinlabel {$c'$} at 70 667
\pinlabel {$c''$} at 70 566
\pinlabel {$c''$} at 204 600
\pinlabel {$c$} at 204 775
\pinlabel {$c$} at 314 735
\pinlabel {$c''$} at 314 640
\endlabellist
\centering
\includegraphics[width=0.4\textwidth,scale=1.2]{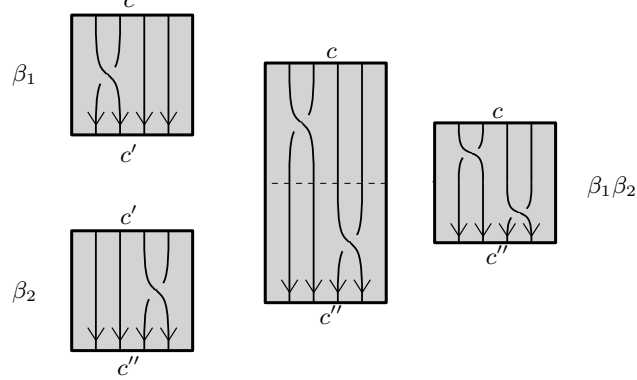}
\vspace*{4mm}
\caption{A $(c,c')$-braid $\beta_1$, a $(c',c'')$-braid $\beta_2$ and their composition, the $(c,c'')$-braid $\beta_1\beta_2$.}
\label{fig:CompositionTwistedOriente}
\end{figure}

A braid~$\beta$ is{\em~$\mu$-colored\/} if each of its components is assigned (via a surjective map) an element in~$\{1,2,\dots,\mu\}$, resulting in a sequence $c=(c_1,c_2,\dots,c_n)$ of integers.  
A~$\mu$-colored braid induces a coloring on the punctures of $D^2 \times \lbrace 0, 1 \rbrace$. For emphasis, we shall denote the resulting punctured disks by~$D_c$ and~$D_{c'}$, and call a~$\mu$-colored braid a~$(c,c')$-braid, where~$c$ and~$c'$ are the sequences of~$ 1, 2,\dots,\mu$ induced by the coloring of the braid. Two colored braids are isotopic if the underlying isotopy is color preserving, and we shall denote by~$\id_c$ the isotopy class of the trivial~$(c,c)$-braid. The composition of a $(c,c')$-colored braid $\beta_1$ with a $(c',c'')$-colored braid $\beta_2$ is the $(c,c'')$-braid $\beta_1 \beta_2$ (see Figure \ref{fig:CompositionTwistedOriente}). Thus, for any sequence~$c$, the set~$B_c$ of isotopy classes of~$(c,c)$-braids is a group which interpolates between the braid group $B_n=B_{(1,1,\dots, 1)}$ and the pure braid group~$P_n=B_{(1,2,\dots,n)}$. Finally, the closure of a $\mu$-colored braid $\beta \in B_c$ is the $\mu$-colored link $\hat{\beta}$ obtained from $\beta$ by adding colored parallel strands in $S^3 \setminus (D^2 \times [0,1]).$

\subsection{Twisted homology groups}
\label{sub:homology}
Next, relying on the exposition of \citep{KL, Friedl-Kim, Friedl-Cha}, we review the construction of twisted homology. Some computations are then made in the case of the punctured disk $D_n$.

Let~$X$ be a connected CW complex endowed with a basepoint $z$ and let $Y$ be a connected CW subspace of $X$. We denote by $p: \tilde{X} \rightarrow X$ the universal cover of $X$ and write $\tilde{Y}=p^{-1}(Y)$. If~$R$ is an integral domain, then any representation~$\varphi: \pi_1(X,z) \rightarrow GL_k(R)$ induces a $(R,\mathbb{Z}[\pi_1(X,z)])$-bimodule structure on~$R^k$, where the right action is given by right multiplication on row vectors. On the other hand, the left action of~$\pi_1(X,z)$ on~$\tilde{X}$ gives rise to a left~$\mathbb{Z}[\pi_1(X,z)]$-module structure on the cellular chain complex~$C_*(\tilde{X},\tilde{Y})$. Following Kirk-Livingston \citep{KL}, the \textit{twisted chain complex} of the pair~$(X,Y)$ is the chain complex of $R$-modules
$$ C_*^\varphi(X,Y;R^k)=R^k \otimes_{\mathbb{Z}[\pi_1(X,z)]}C_*(\tilde{X},\tilde{Y}),$$
and the corresponding \textit{twisted homology groups}~$H_i^\varphi(X,Y;R^k)$ of~$(X,Y)$ are the $R$-modules obtained by taking the homology of this chain complex. If $Y$ is empty, then we write $H_*^\varphi(X;R^k)$ instead of $H_*^\varphi(X,\emptyset;R^k).$ Observe that when $\varphi$ is the trivial one-dimensional representation, then the twisted homology of $X$ coincides with the usual homology of $X$ with coefficients in $R$. 

Following the notation of Friedl-Kim \citep{Friedl-Kim}, we let $H_*^\rho(Y \subset X;R^k)$ be the homology of the chain complex $R^k \otimes_{\mathbb{Z}[\pi_1(X,z)]} C_*(\tilde{Y}).$ Standard arguments then give rise to the long exact sequence 
$$
 \dots \rightarrow H_i^\rho(Y \subset X;R^k) \rightarrow H_i^\rho( X;R^k) \rightarrow H_i^\rho(X,Y;R^k) \rightarrow \dots 
$$
If the basepoint $z$ lies in $Y$, then the composition $\pi_1(Y,z) \rightarrow \pi_1(X,z) \stackrel{\varphi}{\rightarrow} GL_k(R)$ induces twisted homology groups $H_i^\varphi(Y;R^k)$ using the universal cover of $Y$. As $H_i^\varphi(Y \subset X;R^k)$ is isomorphic to $H_i^\varphi(Y;R^k)$ (\cite[Lemma $2.1$]{Friedl-Kim}), the previous long exact sequence yields the long exact sequence
$$
 \dots \rightarrow H_i^\varphi(Y;R^k) \rightarrow H_i^\varphi( X;R^k) \rightarrow H_i^\varphi(X,Y;R^k) \rightarrow \dots 
$$
As the isomorphism type of $H_i^\varphi(X,Y;R^k)$ does not depend on the choice of the basepoint, we will drop it from the notation.

\begin{remark}
\label{rem:shapiro}
If $\rho: \pi_1(X) \rightarrow GL_k(R)$ is a representation,~$N$ is a normal subgroup of~$\pi_1(X)$ and~$\psi: \pi_1(X) \rightarrow \pi_1(X)/N$ is the quotient map, then $\pi_1(X)$ acts on~$R^k \otimes_R R[\pi_1(X)/N]$ via 
$$ (u \otimes v) \cdot \gamma=u \rho(\gamma) \otimes v \psi(\gamma), $$
where $\gamma \in \pi_1(X)$, $u \in R^k$ and $v \in R[\pi_1(X)/N]$. If $R^k \otimes_R R[\pi_1(X)/N]$ is identified with $R[\pi_1(X)/N]^k$ and $\rho \otimes \psi\colon \pi_1(X) \rightarrow GL_k(R[\pi_1(X)/N])$ denotes the resulting representation, then one can form the twisted homology groups $H_*^{\rho \otimes \psi}(X;R[\pi_1(X)/N]^k).$ 

On the other hand, if~$X_N$ is the cover of~$X$ associated to the subgroup~$N$, then the universal covering $\tilde{X}$ also covers $X_N$ with deck transformation group $N$. Restricting the representation $\rho$ to the subgroup $\pi_1(X_N)=N$, one can then consider the twisted homology groups $H_*^\rho(X_N;R^k).$ In this setting, it is known (\citep[Chapter $5$]{DavisKirk}) that
$$H_*^{\rho \otimes \psi}(X; R[\pi_1(X)/N]^k) \cong H_*^\rho(X_N;R^k).$$
In particular, if $\rho$ is the trivial one-dimensional representation, then $H_*^{\rho \otimes \psi}(X;R[\pi_1(X)/N]^k)$ coincides with the (untwisted) homology of $X_N$ with coefficients in $R$.
\end{remark}

\begin{lemma}
\label{lem:RelativeTwisted}
If $z \in D_n$, then the $R$-module~$H_1^\varphi(D_n,z;R^k)$ is free of rank~$nk$.
\end{lemma}

\begin{proof}
The punctured disk~$D_n$ is homotopy equivalent to the wedge of the~$n$ loops representing the generators of $\pi_1(D_n)$ described in Subsection (\ref{sub:braids}). Choose a cellular decomposition of this latter space $X$ consisting of the~$0$-cell~$z$ (the basepoint of the wedge) and one~$1$-cell~$x_i$ for each loop. For $i=1,2,\dots,n$, let $\tilde{x}_i$ be the lift of $x_i$ starting at an (arbitrary) fixed lift of $z$. With this cell structure, the twisted chain complex of $(X,z)$ is 
$$C_1^\varphi(X,z;R^k)=R^k \otimes_{\mathbb{Z}[\pi_1(X)]} C_1(\tilde{X},\tilde{z}) = R^k \otimes_{\mathbb{Z}[\pi_1(X)]} \bigoplus_{i=1}^n \mathbb{Z}[\pi_1(X)]\tilde{x}_i \cong \bigoplus_{i=1}^n R^k \tilde{x}_i.$$
As the chain group~$C_0(\tilde{X}, \tilde{z})$ vanishes,~$H_1^\varphi(D_n,z;R^k)=C_1^\varphi(X,z;R^k)$ and the claim follows. 
\end{proof}

 If~$\varphi$ is a representation of the fundamental group of a space~$Y$, then any map~$f:X \rightarrow Y$ induces a homomorphism
$$ H_i^{\varphi f_*}(X;R^k) \rightarrow H_i^{\varphi}(Y;R^k)$$
on the twisted homology groups, where~$f_*$ is the homomorphism induced by~$f$ on the level of the fundamental groups. 

\begin{example}
\label{ex:braidInduced}
Fix a basepoint $z \in \partial D_n$. Let~$h_\beta: D_n \rightarrow D_n$ be a homeomorphism representing a braid~$\beta \in B_n$. As~$h_\beta$ fixes the boundary of the disk, it lifts uniquely to a homeomorphism~$\tilde{h}_\beta\colon \tilde{D}_n \rightarrow \tilde{D}_n$ which preserves a fixed lift of $z$. Up to homotopy, this lift depends uniquely on the isotopy class of~$h_\beta$ and consequently the map induced on the chain group~$C_1(\tilde{D}_n,\tilde{z})$ depends uniquely on the braid~$\beta$. Therefore each colored braid $\beta$ induces a well-defined homomorphism
$$ H_1^{\beta_*\varphi}(D_n;R^k) \rightarrow H_1^\varphi(D_n;R^k)$$
on twisted homology. 

As the braid group acts by right automorphisms on the free group $\pi_1(D_n)$, the left $\mathbb{Z}[\pi_1(D_n)]$-module~$C_1(\tilde{D}_n,\tilde{z})$ inherits a right action of the braid group. As in Subsection (\ref{sub:braids}), the composition of automorphisms on $C_1(\tilde{D}_n,\tilde{z})$ will be read from \textit{left to right.}

Finally, observe that if~$H$ denotes the free abelian group on~$t_1,t_2,\dots, t_\mu$, then the epimorphism~$\psi_c: \pi_1(D_n) \rightarrow H, \ x_i \mapsto t_{c_i}$ satisfies $\beta_*\psi_c =\psi_c$ precisely when $\beta$ is a $(c,c)$-colored braid.
\end{example}

\subsection{Torsion of chain complexes}
\label{sub:torsion}
Next, we review briefly the definition of the torsion of a chain complex. Standard references include \citep{Milnor} and \citep{Turaev}.

Given two bases ${\bf c},{\bf c}'$ of a finite dimensional vector space over a field~$F$, let~$[{\bf c} / {\bf c}'] \in F \setminus \lbrace 0 \rbrace $ be the determinant of the matrix expressing the vectors of the basis ${\bf c}$ as a linear combination of vectors in ${\bf c}'$. Let 
$C= \left( 0 \rightarrow C_m \rightarrow C_{m-1}\rightarrow \dots \rightarrow C_0 \rightarrow 0 \right)$
be a chain complex of vector spaces over~$F$ such that for $i=1,2,\dots, m$ each $C_i$ has a distinguished basis ${\bf c}_i$. If $C$ is not acyclic, then we set $\tau(C)=0$. Otherwise, let ${\bf b}_i$ be a sequence of vectors in $C_i$ such that $\partial_{i-1}({\bf b}_i)$ forms a basis of $\text{Im}(\partial_{i-1})$. Clearly the sequence~$\partial_i({\bf b}_{i+1}){\bf b}_i$ is a basis of $C_i$. The \textit{torsion} of the based chain complex~$C$ is defined as
$$ \tau(C)=\prod_{i=0}^m [\partial_i({\bf b}_{i+1}){\bf b}_i/{\bf c}_i]^{(-1)^{i+1}} \in F \setminus \lbrace 0 \rbrace.$$
It turns out that $\tau(C)$ depends on the choice of basis for~$C_i$ but does not depend on the choice of ${\bf b}_i$. For the proof of the next proposition, see \cite[Theorem $0.1.1$]{TuraevArticle}
\begin{proposition}
\label{prop:multtorsion}
Let~$0 \rightarrow C' \rightarrow C \rightarrow C'' \rightarrow 0$ be a short exact sequence of finite-dimensional chain complexes over $F$. Assume that~$C'$ or $C''$ is acyclic, and that~$C_i, C_i', C_i''$ have distinguished bases~${\bf c}_i, {\bf c}_i', {\bf c}_i''$ such that~$[{\bf c}_i/{\bf c}_i'{\bf c}_i'']=1.$ Then~$\tau(C)=\pm \tau(C')\tau(C'').$
\end{proposition}

\subsection{Twisted torsion of links}
\label{sub:twistedtorsion}
Following closely Cha-Friedl \citep{Friedl-Cha}, we describe the twisted torsion of $CW$ complexes and links. Finally, we review Fox calculus and how it allows explicit computations of the twisted torsion (\citep{Kitano,Wada}).

Let~$X$ be a finite CW-complex, let~$\rho\colon \pi_1(X) \rightarrow GL_k(R)$ be a representation and let~$\psi:  \pi_1(X) \rightarrow H$ be an epimorphism onto a free abelian group $H$. If $R$ comes with an involution, we endow $R[H]$ with the involution $\overline{rh}=\overline{r}h^{-1}$ for $r \in R$ and $h \in H.$ This extends to an involution on $Q(H),$ the quotient field of $R[H].$
The homomorphisms $\rho$ and $\psi$ induce an action of $\pi_1(X)$ on $R^k \otimes_R R[H]$ by setting
$$ (u \otimes v) \cdot \gamma=u \rho(\gamma) \otimes v \psi(\gamma), $$
where $\gamma \in \pi_1(X)$, $u \in R^k$ and $v \in R[H]$. If  $R^k \otimes_R R[H]$ is identified with $R[H]^k$ and $\rho \otimes \psi\colon \pi_1(X) \rightarrow GL_k(R[H])$ denotes the resulting representation, then one can form the twisted homology groups $H_*^{\rho \otimes \psi}(X;R[H]^k)$. Moreover, as the representation $\rho \otimes \psi $ induces a representation $\rho \otimes \psi: \pi_1(X) \rightarrow GL_k(Q(H))$, one may also consider the $Q(H)$-vector spaces $H_*^{\rho \otimes \psi}(X;Q(H)^k).$

Choose a lift~$\tilde{x}_i^q$ of each~$q$-cell~$x_i^q$ of~$X$ to the universal cover~$\tilde{X}$ and denote by~$e_1,e_2,\dots, e_k$ the canonical basis of~$Q(H)^k.$ This yields a basis~$ \lbrace \tilde{x}_i^q  \otimes e_j \rbrace$ over $Q(H)$ for~$C_q^{\rho \otimes \psi}(X;Q(H)^k).$
 
\begin{definition}
If the chain complex~$C_*^{\rho \otimes \psi}(X;Q(H)^k)$ is acyclic, then the \textit{twisted torsion}~
$$\tau^{\rho \otimes \psi}(X) \in Q(H)\setminus \lbrace 0 \rbrace~$$ of~$X$ is the torsion of the chain complex~$C_*^{\rho \otimes \psi}(X;Q(H)^k)$. If~$C_*^{\rho \otimes \psi}(X;Q(H)^k)$ is not acyclic, then we set $\tau^{\rho \otimes \psi}(X)=0.$
\end{definition}

It is known (\citep{Milnor, Turaev, Kitano, KL, FKK, Friedl-Cha}) that the twisted torsion~$\tau^{\rho\otimes \psi}(X)$ is well-defined up to multiplication by an element in~$\pm \det(\rho \otimes \psi(\pi_1(X)))$ and is invariant under simple homotopy. Since $\det(\rho \otimes \psi(\pi_1(X)))$ is contained in $\det(\rho(\pi_1(X))) \cdot H$, one often considers~$\tau^{\rho\otimes \psi}(X)$ up to multiplication by $\pm d h$ for $d \in \det(\rho(\pi_1(X))) $ and $h \in H$. By Chapman's theorem \citep{Chapman},~$\tau^{\rho \otimes \psi}(X)$ only depends on the homeomorphism type of~$X$. In particular, when~$M$ is a manifold, one can define~$\tau^{\rho\otimes \psi}(M)$ by picking any~$CW$-structure for~$M$. 

Recall (\citep{C,CF,CC}) that a \textit{colored link} $L$ consists of an oriented link $L=L_1 \cup L_2 \cup \dots \cup L_\mu$ together with a surjective map $c$  assigning to each component a color in $ \lbrace 1, 2, \dots ,\mu \rbrace $. The sublink $L_i$ is constituted by the components of $L$ with color $i$, for $i=1,2,\dots,\mu$. If~$H$ is the free abelian group (written multiplicatively) on~$t_1,t_2,\dots, t_\mu$ and $X_L$ is the exterior of $L$, let~$\psi_c$ be the epimorphism  $ \pi_1(X_L)\rightarrow H, \gamma \mapsto t_1^{lk(\gamma,L_1)}\cdots t_\mu^{lk(\gamma,L_\mu)}$, where $lk$ denotes the linking number. The \textit{twisted torsion~$\tau^{\rho \otimes \psi_c}(L)$ of the colored link~$L$} is then the twisted torsion of the exterior of $L$.  When the context is clear, we shall drop the $\psi$'s in the notation of twisted homology and twisted torsion. Variations of the following lemma are well-known (\citep{Turaev, KL, Friedl-Kim, Friedl-Cha}),  but we give a proof for completeness.

\begin{lemma}
\label{lem:nonzero}
If the $R[H]$-module~$H_1^{\rho}(X_L;R[H]^k)$ is torsion, then $C_*^{\rho}(X_L;Q(H)^k)$ is acyclic.
\end{lemma}

\begin{proof}
As the link exterior~$X_L$ is homotopy equivalent to a~$2$-complex,~$H_3^{\rho}(X_L;R[H]^k)$ vanishes. If $X_L^\psi$ denotes the covering of $X_L$ corresponding to the kernel of $\psi$, then Remark \ref{rem:shapiro} implies that~$H_0^{\rho}(X_L;R[H]^k)$ is isomorphic to $H_0^\rho(X_L^\psi;R^k)$. Since $\psi$ is surjective, $X_L^\psi$ is connected and consequently~$H_0^{\rho}(X_L;R[H]^k)$ is $R[H]$-torsion. An Euler characteristic argument then shows that~$H_2^{\rho}(X_L;R[H]^k)$ is torsion over~$R[H]$. As all the twisted homology groups of~$X_L$ are torsion over~$R[H]$, the chain complex~$C^{\rho}_*(X_L;Q(H)^k)$ is acyclic and the claim follows.
\end{proof}

Next, following Wada \citep{Wada} and Kitano \citep{Kitano}, we shall recall how $\tau^{\rho}(X)$ can be computed via Fox calculus. Given a free group~$F$, the \textit{Fox derivative} (first introduced by Fox \citep{Fox})
$$\frac{\partial}{\partial x_j}: \mathbb{Z}[F] \rightarrow \mathbb{Z}[F]~$$
 is the linear extension of the map defined on elements of~$F$ by 
$$
\frac{\partial x_i}{\partial x_j}=\delta_{ij}, \ \ \ \ \ \ \ \ \ \frac{\partial x_i^{-1}}{\partial x_j}=-\delta_{ij} x_i^{-1}, \ \ \ \ \ \ \ \ \  \frac{\partial (uv)}{\partial x_j}=\frac{\partial u}{\partial x_j}+u\frac{\partial v}{\partial x_j}.
$$
Choose a cellular decomposition of the CW complex~$X$ with one~$0$-cell~$v$,~$n$ oriented~$1$-cells labeled~$x_1, x_2,\dots,x_n$ having all their endpoints identified with~$v$ to form~$n$ loops, and~$m$ oriented~$2$-cells~$c_1, c_2,\dots, c_{m}$ with each~$\partial c_i$ glued to the~$1$-cells according to a word~$r_i$. The fundamental group of $X$ then admits a presentation with generators $x_1, x_2, \dots, x_n$ and relators $r_1,r_2, \dots, r_m.$ Let~$\tilde{v}, \tilde{x}_i$ and~$\tilde{c}_i$ be corresponding lifts to the universal cover~$p: \tilde{X} \rightarrow X$.

Let~$F_n$ be the free group on~$x_1,x_2,\dots,x_n$ and let $pr: \mathbb{Z}[F_n] \rightarrow \mathbb{Z}[\pi_1(X)]$ denote the ring homomorphism induced by the quotient map. The chain group~$C_1(\tilde{X},p^{-1}(v))$ is generated by the~$\tilde{x}_i$, and if~$w$ is a word in the~$x_i$, then its lift~$\tilde{w}$ (viewed as a~$1$-chain in the universal cover) can be written as
$$ \tilde{w}=\sum_{j=1}^n pr \left( \frac{\partial w}{\partial
x_j} \right) \tilde{x}_j.$$
Since the boundary map~$\partial_2$ of the chain complex~$C_*(\tilde{X})$ sends~$\tilde{c}_i$ to the lift of~$r_i$ beginning at~$\tilde{v}$, the previous equation specializes to
$$ \partial_2(\tilde{c}_i)=\sum_{j=1}^n  pr \left( \frac{\partial r_i}{\partial x_j} \right) \tilde{x}_j.$$
Following a common convention (\citep{Turaev}, \citep{KL}), we shall assume that the elements in the chain complex $C_*(\tilde{X})$ of free left $\mathbb{Z}[\pi_1(X)]$-modules are row vectors and that the matrices of the differentials act by right multiplication. Consequently, $\partial_2$ is represented by the $((n-1)\times n)$ matrix whose $(i,j)$-coefficient is 
$ pr \left( \frac{\partial r_i}{\partial
x_j} \right). $

Slightly abusing notation, we shall also denote by~$\rho \otimes \psi$ the composition of the map~$ pr: \mathbb{Z}[F_n] \rightarrow \mathbb{Z}[\pi_1(X)]$ with the map~$ \rho \otimes \psi: \mathbb{Z}[\pi_1(X)] \rightarrow M_k(R[H])$. The boundary map~$\id \otimes \partial_2$ in the twisted chain complex~$C_*^\rho(X,v;R[H]^k)$ is then represented by the $((n-1) \times n)$ matrix~$A$ whose~$(i,j)$ coefficient is 
$$ \rho \otimes \psi \left(\frac{\partial r_i}{\partial x_j}\right) \in M_k(R[H]).$$
Assume (for the sake of exposition) that $\pi_1(X)$ admits a deficiency one presentation with $n$ generators. For~$j=1,2,\dots,n,$ regard the matrix~$A_j$, obtained by removing the~$j$-th column of~$A$, as a~$((n-1)k \times (n-1)k)$ matrix with coefficients in~$R[H]$.  The \textit{Wada invariant} of~$X$ is 
$$W(X)=\frac{\det(A_j)}{\det\left( (\rho \otimes \psi)(x_j-1) \right)}.$$
Wada \citep{Wada} proved that~$\det\left( (\rho \otimes \psi)(x_j-1) \right)$ is non-zero for~$j=1,2,\dots,n$ and that~$W$ is independent of the choice of the index~$j$. Kitano \citep{Kitano} showed that Wada's invariant coincides with the twisted torsion of~$X$:
$$W(X)= \pm d h \ \tau^\rho(X)$$
for some $d \in \det(\rho(\pi_1(X)))$ and $h \in H$. In practice, we shall use Wada's invariant in order to compute the twisted torsion. Note that if $\pi_1(X)$ admits the relation $r=s$, then Fox calculus yields
$$ \frac{\partial (rs^{-1})}{\partial x_i}=\frac{\partial r}{\partial x_i}-rs^{-1}\frac{\partial s}{\partial x_i},$$
and since $\rho \otimes \psi(rs^{-1})=I_k$, one obtains
$$ \rho \otimes \psi \left( \frac{\partial (rs^{-1})}{\partial x_i} \right)=\rho \otimes \psi \left(\frac{\partial r}{\partial x_i} \right)-\rho \otimes \psi  \left( \frac{\partial s}{\partial x_i} \right)=\rho \otimes \psi \left(\frac{\partial (r-s)}{\partial x_i} \right). $$
Consequently the Fox derivatives of the relator $rs^{-1}$ can be computed by considering the element $r-s$ of the group ring $\mathbb{Z}[F_n].$

\begin{example}
\label{ex:trefoil}
Let~$T$ be the trefoil knot. The group~$\pi_1(S^3 \setminus T)$ admits a presentation with two generators~$ x_1, x_2~$ and a unique relation~$x_1x_2x_1=x_2x_1x_2$. If $r$ denotes  $x_1x_2x_1-x_2x_1x_2$, then Fox calculus shows that
$$ \frac{\partial r}{\partial x_1}=1-x_2+x_1x_2 $$
and
$$ \frac{\partial r}{\partial x_2}=-1-x_1+x_2x_1.$$
Let~$\rho\colon \pi_1(S^3 \setminus T) \rightarrow GL_2(\mathbb{Z}[s^{\pm 1}])$ be the representation given by 
\[
 \rho(x_1)=
\begin{pmatrix}
-s & 1 \\
0 & 1
\end{pmatrix}
, \ \
 \rho(x_2)=
\begin{pmatrix}
1 & 0 \\
s & -s
\end{pmatrix}.
\]
If~$\psi: \pi_1(S^3 \setminus T) \rightarrow \mathbb{Z}=\langle t \rangle$ is the abelianization homomorphism sending~$x_i$ to~$t$ for~$i=1,2$, then a short computation shows that
$$ \det \left(\rho \otimes \psi \left( \frac{\partial r}{\partial x_1}\right) \right)=\det
\begin{pmatrix}
1-t& -st^2 \\
-st+st^2 & 1+st-st^2 
\end{pmatrix}
=(1-t)(1+st)(1-st^2)$$
and 
$$ \det((\rho \otimes \psi) (1-x_1))=\det 
\begin{pmatrix}
1+st & -t \\
0 & 1-t 
\end{pmatrix}
= (1-t)(1+st).
$$
Therefore the twisted torsion of $T$ is~$\tau^{\rho}(T)=1-st^2$ up to the indeterminacy $\pm d h $ (with $d \in \det(\rho(\pi_1(S^3 \setminus T)))$ and $h \in H$) which in this case is $ \pm s^m t^n$, where $m,n \in \mathbb{Z}$.
\end{example}

Note that if~$\rho$ is the trivial one-dimensional representation, $\mu=1$, and~$\psi: \pi_1(X_L) \rightarrow \mathbb{Z}=\langle t \rangle$ is the homomorphism sending each meridian of a link~$L$ to~$t$, then~$(t-1)\tau^{\rho}(L)$ is equal to the Alexander polynomial~$\Delta_L(t)$ of~$L$. On the other hand, if $L$ has $\mu=n \geq 2$ components and~$\psi$ is the abelianization homomorphism, then the twisted torsion is equal to the multivariable Alexander polynomial~$\Delta_L(t_1,\dots,t_n)$ of~$L$.

\section{The twisted Burau map and the twisted Alexander polynomial}
\label{sec:results}

\subsection{The twisted Burau map}
\label{sub:Burau}
In this subsection we define the twisted Burau map and show how to compute it using Fox calculus.

Fix a sequence~$c=(c_1,c_2,\dots,c_n)$ of elements in~$\lbrace 1,2,\dots, \mu \rbrace$ and a representation~$\rho\colon \pi_1(D_n) \rightarrow GL_k(R)$. If~$H$ denotes the free abelian group (written multiplicatively) on~$t_1,t_2,\dots, t_\mu$, then we let~$\psi_c: \pi_1(D_n) \rightarrow H$ be the epimorphism defined by~$x_i \mapsto t_{c_i}$. Given a basepoint~$z \in \partial D_n$ of the punctured disk~$D_n$, we saw in Example \ref{ex:braidInduced} that each colored braid~$\beta \in B_c$ induces a well-defined homomorphism
$$\mathcal{B}_\rho(\beta)\colon H_1^{\beta_*\rho}(D_n,z;R[H]^k) \rightarrow H_1^{\rho}(D_n,z;R[H]^k)$$
on twisted homology.
Using the same notations as in the proof of Lemma \ref{lem:RelativeTwisted}, we shall call the basis resulting from the isomorphism
$$H_1^{\rho}(D_n,z;R[H]^k) \cong \bigoplus_{i=1}^n R[H]^k \tilde{x}_i$$
the \textit{good basis} of~$H_1^{\rho}(D_n,z;R[H]^k).$ With respect to the good bases of~$H_1^{\beta_*\rho}(D_n,z;R[H]^k)~$ and~$H_1^{\rho}(D_n,z;R[H]^k)$, the homomorphism~$\mathcal{B}_\rho(\beta)$ gives rise to a~$kn \times kn$ matrix with coefficients in~$R[H].$

\begin{definition}
The \textit{twisted Burau map} 
$$ \mathcal{B}_\rho\colon B_c \rightarrow GL_{nk}(R[H])$$
sends a colored braid~$\beta$ to the matrix~$\mathcal{B}_\rho(\beta) \in GL_{nk}(R[H])$ defined above. 
\end{definition}

The next lemma shows that while the twisted Burau map is generally not a representation, it is nevertheless determined by the generators of $B_c$ (compare with \citep[Equation~$(12)$]{Boden}).

\begin{lemma}
\label{lem:cocycle}
If~$\beta, \gamma \in B_c$ are two $\mu$-colored braids, then the equation
$$\mathcal{B}_{\rho}(\beta \gamma )=\mathcal{B}_{\gamma_*\rho}(\beta)\mathcal{B}_{\rho}(\gamma )$$
holds for each representation~$\rho$ of the free group~$\pi_1(D_n)$. 
\end{lemma}

\begin{proof}
The composition 
$$  H_1^{\beta_*\gamma_*\rho}(D_{n},z;R[H]^k)  \stackrel{\mathcal{B}_{\gamma_*\rho}(\beta)}{\longrightarrow} H_1^{\gamma_*\rho}(D_{n},z;R[H]^k) \stackrel{\mathcal{B}_{\rho}(\gamma )}{\longrightarrow} H_1^{\rho}(D_n,z;R[H]^k)$$
of the maps induced by~$\beta$ and~$\gamma $
coincides with the map induced by~$\beta \gamma $. As the braid group acts on the twisted homology from the right, and composition is read from left to right (see Example \ref{ex:braidInduced}), the claim follows immediately.
\end{proof}
 
\begin{example}
If $R=\mathbb{Z}$ and $\rho$ is the trivial one-dimensional representation, then one gets a homomorphism $\mathcal{B}_\rho\colon B_c \rightarrow GL_n(\mathbb{Z}[t_1^{\pm 1},\dots,t_\mu^{\pm 1}])$ which coincides with the colored Gassner representation. In particular if $c=(1,1,\dots,1)$, then $\mathcal{B}_\rho$ is the classical Burau representation, while if $c=(1,2,\dots,n)$, then $\mathcal{B}_\rho$ is the Gassner representation (see also Examples \ref{ex:sigma1unred} and \ref{ex:gassner}).
\end{example} 
 
\begin{proposition}
\label{prop:BurauFox}
Let $\beta \in B_c$ be a colored braid. Consider the~$(n \times n)$-matrix~$A$ whose~$(i,j)$  component is 
~$$(\rho \otimes \psi_c)  \left(\frac{\partial  (x_i \beta)  }{\partial x_j} \right) \in M_k(R[H]).~$$
If one views~$A$ as a $(nk \times nk)$-matrix with coefficients in~$R[H]$, then $\mathcal{B}_\rho(\beta)$ is equal to $A$.
\end{proposition}

\begin{proof}
Fix a lift of $z$ to the universal cover. Given a homeomorphism~$h_\beta$ representing a braid~$\beta$, let~$\tilde{h}_\beta$ be the map induced by the lift of~$h_\beta$ on the chain group $C_1(\tilde{D}_n,\tilde{z})$. As
$H_1^{\rho}(D_n,z;R[H]^k)=R[H]^k \otimes_{\mathbb{Z}[\pi_1(D_n)]} C_1(\tilde{D}_n,\tilde{z})$, the twisted Burau map is given by the homomorphism $\mathcal{B}_\rho(\beta)=\id \otimes \tilde{h}_{\beta}.$ Clearly~$ \tilde{x}_i \tilde{h}_{\beta }$ is the lift of a loop representing~$x_i \beta$ to the universal cover.  Fox calculus then shows that on the chain group level
\begin{equation*}
\tilde{x}_i \tilde{h}_{\beta}   =\sum_{j=1}^n \frac{\partial (x_i \beta)}{\partial x_j} \tilde{x}_j.
\end{equation*}
As we view elements of the left $\mathbb{Z}[\pi_1(D_n)]$-module $C_1(\tilde{D}_n,\tilde{z})$ as row vectors, $\tilde{h}_\beta$ is represented by the $(n \times n)$ matrix whose $(i,j)$ component is $\frac{\partial (x_i \beta)}{\partial x_j}$. The claim now follows from the right $\mathbb{Z}[\pi_1(D_n)]$-module structure of $R[H]^k$.
\end{proof}

The definitions and propositions of this subsection can easily be adapted to include $(c,c')$-braids. Namely, to each $(c,c')$-braid, one may associate a $(kn \times kn)$-matrix which coincides with the $(n \times n)$-matrix whose $(i,j)$ coefficient is ~$(\rho \otimes \psi_{c'})  \left(\frac{\partial  (x_i \beta)  }{\partial x_j} \right) \in M_k(R[H]).$

\begin{example}
\label{ex:sigma1unred}
 A short computation involving Fox calculus shows that
$$
\frac{\partial (x_i \sigma_i ) }{\partial x_i}=\frac{\partial(x_i x_{i+1} x_i^{-1})}{\partial x_i}= 1-x_i x_{i+1} x_i^{-1}, 
$$
and 
$$
\frac{\partial (x_i \sigma_i  )}{\partial x_{i+1}}=\frac{\partial(x_i x_{i+1} x_i^{-1}) }{\partial x_{i+1}}= x_i.
$$
Consequently, with respect to the good bases, the twisted Burau map of~$\sigma_i$ (viewed as a $(c,c')$-colored braid) is given by
\[
 \mathcal{B}_\rho(\sigma_i)
= I_{(i-1)k} \oplus
\begin{pmatrix}
I_k-\rho(x_i x_{i+1} x_i^{-1})t_{c'_{i+1}} & \rho(x_i)t_{c'_{i}} \\
I_k & 0
\end{pmatrix}
 \oplus I_{(n-i-1)k}.
\]
If~$\rho$ is the trivial one-dimensional representation,~$\mu=1$ and $R=\mathbb{Z}$, then one recovers the (unreduced) Burau representation~$\mathcal{B}_t$. 
 \end{example}
 
\begin{example}
\label{ex:gassner}
Write $c=(1,2), \ c'=(2,1)$ and decompose the pure braid $\sigma_1^2 \in P_2=B_c$ as  $\sigma_1 \sigma_1'$, where $\sigma_1$ is viewed as a $(c,c')$ braid and  $\sigma_1'$ is the braid $\sigma_1$ viewed as a $(c',c)$-braid. Assume that $R=\mathbb{Z}$ and that $\rho$ is the trivial one dimensional representation. In this case, Lemma \ref{lem:cocycle} yields
$$
\mathcal{B}_\rho(\sigma_1^2)=\mathcal{B}_\rho(\sigma_1)\mathcal{B}_\rho(\sigma_1')=
\begin{pmatrix}
1-t_1 & t_2 \\
1 & 0 
\end{pmatrix}
\begin{pmatrix}
1-t_2 & t_1 \\
1 & 0 
\end{pmatrix}
=\begin{pmatrix}
1-t_1+t_1t_2 & t_1(1-t_1) \\
1-t_2& t_1 
\end{pmatrix},
$$
which is the Gassner matrix of the pure braid $\sigma_1^2.$ On the other hand, the action of $\sigma_1^2$ on the free group $F_2$ is given by
\begin{align*}
&x_1 \sigma_1^2=(x_1x_2x_1^{-1})\sigma_1=(x_1x_2x_1^{-1})x_1(x_1x_2^{-1}x_1^{-1})=x_1x_2x_1x_2^{-1}x_1^{-1},  \\
&x_2 \sigma_1^2=x_1 \sigma_1=x_1x_2x_1^{-1},
\end{align*} 
and Fox calculus yields
\begin{align*}
&\frac{\partial (x_1 \sigma_1^2 ) }{\partial x_1}=\frac{\partial(x_1x_2x_1x_2^{-1}x_1^{-1}) }{\partial x_1}=1+x_1x_2-x_1x_2x_1x_2^{-1}x_1^{-1} ,\\
&\frac{\partial (x_1 \sigma_1^2 ) }{\partial x_2}=\frac{\partial(x_1x_2x_1x_2^{-1}x_1^{-1})}{\partial x_2}=x_1(1-x_2x_1x_2^{-1}).
\end{align*}
Consequently, applying $\psi_c$, one obtains the same matrix as above. 
\end{example}

\subsection{The reduced twisted Burau map}
\label{sub:reduced}
In this subsection, we shall generalize the definition of the reduced Burau  representation to the twisted setting.

\begin{proposition}
\label{prop:Birman}
Fix a basepoint~$z \in \partial D_n$. For each braid $\beta$, the twisted Burau map $\mathcal{B}_\rho(\beta)$ fixes a free submodule of $H_1^\rho(D_n,z;R[H]^k)$ of rank~$k$.
\end{proposition}

\begin{proof}
Instead of working with the free generators~$x_1,x_2\dots,x_n$ of~$\pi_1(D_n),$ consider the elements~$g_1, g_2, \dots,g_n,$ where~$g_i=x_1 x_2 \cdots x_i$. The action of the braid group~$B_n$ on this new set of free generators for $\pi_1(D_n)$ is given by
\[
g_j  \sigma_i=
\begin{cases}
 g_j                            & \mbox{if }  j \neq i, \\ 
g_{i+1}g_{i}^{-1} g_{i-1}  & \mbox{if }  j=i \neq 1, \\
g_2 g_1^{-1}            & \mbox{if } j=i=1. \\ 
 \end{cases} 
\] 
Let~$\tilde{g}_i$ be the lift of~$g_i$ starting at a fixed lift of $z$. Using the same argument as in Lemma \ref{lem:RelativeTwisted}, one obtains the splitting 
$$H_1^{\rho }(D_n,z;R[H]^k) = \bigoplus_{i=1}^{n-1} R[H]^k \tilde{g}_i \oplus R[H]^k \tilde{g}_n.~$$
As~$g_n$ is always fixed by the action of the braid group, its lift~$\tilde{g}_n$ is fixed by the lift~$\tilde{h}_\beta$ of a homeomorphism $h_\beta$ representing a braid $\beta$. This concludes the proof of the proposition.
\end{proof}

\begin{definition}
The \textit{reduced twisted Burau map} 
$$ \overline{\mathcal{B}}_\rho\colon B_c \rightarrow GL_{(n-1)k}(R[H])$$
sends a braid~$\beta$ to the restriction~$\overline{\mathcal{B}}_\rho(\beta) \in GL_{(n-1)k}(R[H])$ of the twisted Burau map to the free~$R[H]$-module of rank~$k(n-1)$ given by  the proof of Proposition \ref{prop:Birman}. 
\end{definition} 

Proposition \ref{prop:Birman} immediately yields the following result.

\begin{corollary}
\label{cor:reducedfox}
If~$\tilde{\mathcal{B}}_\rho(\beta)$ denotes the twisted Burau matrix of a braid~$\beta$ with respect to the basis described in the proof of Proposition \ref{prop:Birman}, then 
\[ 
\tilde{\mathcal{B}}_\rho=
 \begin{pmatrix}
\overline{\mathcal{B}}_\rho(\beta) & V \\
0 & I_k
\end{pmatrix}   \]
for some $(k(n-1) \times k)$-matrix~$V$.
\end{corollary}

The reduced twisted Burau map also satisfies the property of Lemma \ref{lem:cocycle}. Moreover its definition may be easily adapted to include the case of $(c,c')$-braids.

\begin{example}
\label{ex:sigma1}
Combining Proposition \ref{prop:BurauFox} and Corollary \ref{cor:reducedfox}, the reduced twisted Burau map of~$\sigma_i$ (viewed as a $(c,c')$-colored braid) is given by
\begin{align*}
 \overline{\mathcal{B}}_\rho(\sigma_i)
 &=I_{(i-2)k} \oplus 
\begin{pmatrix}
I_k & 0 & 0 \\
\rho(g_{i+1}g_i^{-1})t_{c'_{i+1}} & -\rho(g_{i+1}g_i^{-1})t_{c'_{i+1}}   & I_k \\
0 & 0 & I_k
\end{pmatrix}
 \oplus I_{(n-i-2)k}
 \end{align*}
for~$1<i<n-1$, and for~$\sigma_1$ and~$\sigma_{n-1}$ it is represented by
\begin{align*}
 \overline{\mathcal{B}}_\rho(\sigma_1)
 &=
\begin{pmatrix}
 -\rho(g_{2}g_1^{-1})t_{c'_2}  & I_k \\
 0 & I_k
\end{pmatrix}
 \oplus I_{(n-3)k}, \\
 \overline{\mathcal{B}}_\rho(\sigma_{n-1})
 &=I_{(n-3)k} \oplus 
\begin{pmatrix}
I_k & 0  \\
\rho(g_{n}g_{n-1}^{-1})t_{c'_n} & -\rho(g_{n}g_{n-1}^{-1})t_{c'_n}   \\
\end{pmatrix}.
 \end{align*}
If~$\rho$ is the trivial one-dimensional representation,~$\mu=1$ and $R=\mathbb{Z}$, then one recovers the reduced Burau representation~$\overline{\mathcal{B}}_t$ mentioned in the introduction. 
\end{example}

\begin{example}
\label{ex:sigma13}
We will compute the twisted reduced Burau map of~$\sigma_1^3 \in B_2$ for any representation~$\rho$ in the case when $\mu=1$. Using Example \ref{ex:sigma1},~$\overline{\mathcal{B}}_\rho(\sigma_1)=-\rho(g_2g_1^{-1})t.$ As~$(g_2g_1^{-1})\sigma_1=g_2g_1g_2^{-1}$, it follows from Lemma \ref{lem:cocycle} that
$$\overline{\mathcal{B}}_\rho(\sigma_1^3)=\overline{\mathcal{B}}_{{\sigma_1}_*\rho}\overline{\mathcal{B}}_{{\sigma_1}_*\rho}(\sigma_1)\overline{\mathcal{B}}_{\rho}(\sigma_1)=-\rho(g_2g_1g_2^{-1})\rho(g_2g_1g_2^{-1})\rho(g_2g_1^{-1})t^3 =-\rho(g_2g_1)t^3,$$
and consequently one gets 
$\overline{\mathcal{B}}_\rho(\sigma_1^3)=-\rho(x_1x_2x_1)t^3.$
\end{example}

From the topological viewpoint, the definition of the reduced twisted Burau map is somewhat unsatisfactory compared to the unreduced version. Indeed the basis given by Proposition \ref{prop:Birman} is devoid of topological meaning. This motivates the search of a more intrinsic definition of the reduced map.  If $\beta \in B_c$, is a colored braid, we therefore consider the restriction 
$$\hat{\mathcal{B}}_\rho(\beta)\colon  H_1^{\beta_*\rho}(D_n;R[H]^k) \rightarrow H_1^{\rho}(D_n;R[H]^k)$$
of the twisted Burau map. Even though we do not know whether $H_1^{\beta_*\rho}(D_n;R[H]^k)$ is free, we do know its rank.

\begin{proposition}
Fix a basepoint~$z \in \partial D_n$. 
\label{cor:free}
\begin{enumerate}[(a)]
\item  The dimension of the $Q(H)$-vector space~$H_1^{\rho}(D_n;Q(H)^k)$ is~$k(n-1)$.
\item If~$\mu=1$ and $R$ is a principal ideal domain, then $H_1^{\rho }(D_n;R[H]^k)$ is a free $R[H]$-module of rank $k(n-1).$
\end{enumerate}
\end{proposition}

\begin{proof}
Consider the portion
$$0 \rightarrow H_1^{\rho }(D_n;R[H]^k)  \rightarrow
H_1^{\rho }(D_n,z;R[H]^k) \rightarrow
H_0^{\rho }(z;R[H]^k) \rightarrow H_0^{\rho }(D_n;R[H]^k)$$
of the long exact sequence of the pair $(D_n,z).$ As~$H_1^{\rho}(D_n,z;R[H]^k)$ and~$H_0^{\rho}(z;R[H]^k)$ are free~$R[H]$-modules, and $H_0^{\rho }(D_n;R[H]^k)$ is torsion, the first assertion follows by taking the tensor product with $Q(H)$. The second assertion is an immediate consequence of the following algebraic claim: if $R$ is a principal ideal domain and one has a sequence of $R[t^{\pm 1}]$-modules
$$ 0 \rightarrow K \rightarrow P \rightarrow F,$$
where $P$ and $F$ are free and finitely-generated, then $K$ is also free. To prove the claim, first note that since $R$ is principal, the ring $R[t^{\pm 1}]$ has global dimension $2$ \cite[Theorem $4.3.7$]{Weibel}. Following word for word the proof of \citep[Lemma $3.7$]{CT} (with $R$ instead of $\mathbb{Z}$), it then follows that $K$ is projective. Since $P$ is finitely-generated over the Noetherian ring $R[t^{\pm 1}]$, $K$ is also finitely generated. The conclusion now follows from the fact that if $R$ is a principal ideal domain, then every finitely-generated projective $R[t^{\pm 1}]$-module is free \cite{Swan}.
\end{proof}

\begin{example}
\label{ex:HomolBurau}
Assume that $R=\mathbb{Z}$ and~$\rho$ is the trivial one-dimensional representation. When $\mu=1$, let $D_n^\infty$ be the infinite cyclic covering of $D_n$ corresponding to the kernel of the map $\psi: \pi_1(D_n) \rightarrow \mathbb{Z}=\langle t \rangle$ sending each meridian to $t$. Using Remark \ref{rem:shapiro}, $H_1^{\rho}(D_n;R[H]^k)$ is isomorphic to $H_1(D_n^\infty,\mathbb{Z})$ and the latter is a free $\mathbb{Z}[t^{\pm 1}]$-module of rank $n-1$. In this setting (with respect to appropriate bases), $\hat{\mathcal{B}}_\rho(\beta) $ coincides with the reduced Burau representation (\citep{Bigelow, LongPaton, TuraevFaithful, CT}). On the other hand when $\mu=n$ and $c=(1,2,\dots,n)$, $\hat{\mathcal{B}}_\rho(\beta)$ coincides with the reduced Gassner representation of the pure braid $\beta$ (\citep{KL,CC}).
\end{example}

\begin{remark}
Squier \citep{Squier} observed (via an algebraic computation) that the reduced Burau representation is unitary with respect to a skew-Hermitian form (see also Abdulrahim \citep{Abdulrahim} for a similar observation concerning the Gassner representation). Using the homological description outlined in Example \ref{ex:HomolBurau}, it was later understood (\citep{KL, TuraevFaithful, CT}) that the above-mentioned skew-hermitian form arises from an intersection pairing on $D_n^\infty$. 

If $R$ is a ring with involution and $\rho$ is a unitary representation, then the twisted intersection form (\citep{Levine,KL, Friedl-Cha}) on the free part of $H_1^\rho(D_n;R[H]^k)$ allows us to generalize this observation to the twisted setting. Indeed as the equivariant intersection pairing on the universal covering is preserved by homeomorphisms, the homomorphism~$\hat{\mathcal{B}}_\rho(\beta)$ intertwines the twisted intersection pairings on~$H_1^{\beta_*\rho}(D_n;R[H]^k)$ and~$H_1^{\rho}(D_n;R[H]^k)$.
\end{remark}

In view of Example \ref{ex:HomolBurau}, it is tempting to conclude that $\hat{\mathcal{B}}_\rho$ is equal to the reduced twisted Burau map. Unfortunately, for $k>1$, even if $H_1^{\rho}(D_n;R[H]^k)$ is a free $R[H]$-module (as in Corollary \ref{cor:free}), there is no obvious basis from which one may compute a matrix of $\hat{\mathcal{B}}_\rho$.

\subsection{Relation to the twisted Alexander polynomial}
\label{sub:thm}

Generalizing an idea of Morton \citep{Morton}, we show how the twisted Alexander polynomial can be computed from the reduced twisted Burau map.

\begin{theorem}
\label{thm:main}
Let~$F_n$ be the free group on~$x_1, x_2,\dots, x_n$ and let~$\beta \in B_c$ be a~$\mu$-colored braid  with~$n$ strands. If~$\rho\colon F_n \rightarrow GL_k(R)$ is a representation which extends to~$\pi_1(S^3 \setminus \hat{\beta})$, then
\[
\tau^{\rho}(\hat{\beta})(t_1,t_2,\dots,t_\mu) \det\left(\rho(x_1 x_2 \cdots x_n) t_{c_1}t_{c_2}\cdots t_{c_n}-I_{k}\right)= \pm dh \ \det(\overline{\mathcal{B}}_\rho(\beta)-I_{(n-1)k}),
 \]
for some $d\in \det(\rho(\pi_1(S^3 \setminus \hat{\beta})))$ and $h \in H.$
\end{theorem}

\begin{proof}
Let~$X_\beta$ be the exterior of a braid~$\beta$ in the cylinder~$D^2 \times [0,1].$ The manifold obtained by gluing~$X_\beta$ and~$X_{id_c}$ along~$D_c \sqcup D_c$ is nothing but the exterior of the link~$\hat{\beta} \cup \partial D_c$ in $S^3$. Consequently the exterior~$X_{\hat{\beta}}$ of~$\hat{\beta}$ can be obtained by gluing the solid torus~$D^2 \times \partial D_c$ to~$X_{\hat{\beta} \cup \partial D_c}$ along~$\partial D^2\times \partial D_c$. Identify the free group $F_n$ with $\pi_1(D_c)$ so that the free generators $x_i$ correspond to the loops described in Subsection (\ref{sub:braids}). As in Subsection (\ref{sub:reduced}), the elements $g_1,g_2,\dots,g_n$ then also form a free generating set of~$\pi_1(D_c)$. If $x$ is a meridian of $\partial D_c$, then van Kampen's theorem implies that~$\pi_1(X_{\hat{\beta} \cup \partial D_c})$ admits a presentation where the $n+1$ generators $g_1,g_2,\dots,g_n,x$  are subject to the $n$ relations  $x^{-1} g_i x=g_i \beta$.  The representation~$\rho$ extends to~$\pi_1(X_{\hat{\beta} \cup \partial D_c})$ by setting $\rho(x)=I_k.$ A second application of van Kampen's theorem ensures that this extension of $\rho$ coincides with the representation induced by $\rho$ on $\pi_1(S^3\setminus \hat{\beta}).$ 

Since the generator of~$\pi_1(D^2 \times \partial D_c)$ is~$g_n$, the chain complex~$C_*^\rho(D^2 \times \partial D_c;Q(H)^k)$ is acyclic and the twisted torsion of $D^2 \times \partial D_c$ is equal to~$1/ \det(\rho(g_n)\psi_c(g_n)-I_k)$. Using excision, one observes that $H_2^\rho(X_{\hat{\beta}},X_{\hat{\beta} \cup \partial D_c};R[H]^k)$ is torsion and $H_1^\rho(X_{\hat{\beta}},X_{\hat{\beta} \cup \partial D_c};R[H]^k)$ vanishes. Consequently, the long exact sequence of the pair $(X_{\hat{\beta}},X_{\hat{\beta} \cup \partial D_c})$ with coefficients in $Q(H)^k$ reduces to 
$$ 0  \rightarrow H_1^{\rho}(X_{\hat{\beta} \cup \partial D_c};Q(H)^k)\rightarrow H_1^{\rho}(X_{\hat{\beta}};Q(H)^k)\rightarrow 0.$$
Therefore, if $H_1^\rho(X_{\hat{\beta}};R[H]^k)$ is torsion, then so is $H_1^\rho(X_{\hat{\beta} \cup \partial D_c};R[H]^k)$. In this case, Lemma \ref{lem:nonzero} implies that the twisted chain complexes of $X_{\hat{\beta}}$ and $X_{\hat{\beta} \cup \partial D_c}$ are acyclic over $Q(H)^k.$ Applying Proposition \ref{prop:multtorsion} to the short exact sequence of chain complexes resulting from the decomposition $X_{\hat{\beta}}=X_{\hat{\beta} \cup \partial D_c} \cup (D^2 \times \partial D_c)$ yields
\begin{equation*}
\tau^{\rho \otimes \psi_c}(\hat{\beta})(t_1,\dots,t_\mu) \det\left(\rho(x_1 x_2 \cdots x_n) t_{c_1}t_{c_2}\cdots t_{c_n}-I_{k}\right)=\tau^{\rho \otimes \psi_{c}}(\hat{\beta} \cup \partial D_c)(t_1,\dots,t_\mu).
\end{equation*}
Let us now color the trivial knot~$\partial D_c$ so that~$\hat{\beta} \cup \partial D_c$ becomes a~$\mu+1$-colored link via a sequence~$c'$. If $H'$ is the free abelian group on $t_1,\dots,t_\mu,t_{\mu+1}$ and $\psi_{c'}$ is the homomorphism which coincides with $\psi_c$ on the $g_i$ and  sends $x$ to $t_{\mu+1}$, then one obtains
$$\tau^{\rho \otimes \psi_c}(\hat{\beta} \cup \partial D_c)(t_1,\dots,t_\mu)=\tau^{\rho \otimes \psi_{c'}}(\hat{\beta} \cup \partial D_c)(t_1,\dots,t_\mu,1).$$
Let $A$ be the $(n \times (n+1))$ matrix obtained by performing Fox calculus on the previously described deficiency one presentation of $\pi_1(X_{\hat{\beta} \cup \partial D_c})$. A short computation shows that
\[ \frac{\partial( g_i \beta -x^{-1}g_i x)}{\partial g_j}=\frac{\partial (g_i \beta)}{\partial g_j}-x^{-1}\delta_{ij}. \]
Consequently, using Corollary \ref{cor:reducedfox}, the~$(nk \times nk)$ matrix resulting from the deletion of the~$(n+1)$-th column of $A$ is 
\[ 
A_{n+1}=\begin{pmatrix}
\overline{\mathcal{B}}_\rho(\beta)-t_{\mu+1}^{-1}I_{(n-1)k} & V \\
0 & I_k(1-t_{\mu+1}^{-1})
\end{pmatrix}   \]
for some $(k(n-1) \times k)$-matrix~$V$. As $A_{n+1}$ is an upper triangular block matrix, its determinant is the product of the diagonal blocks. Using Wada's characterization of the twisted torsion and simplifying the~$ I_k(1-t_{\mu+1}^{-1})$ terms, the twisted torsion of $\hat{\beta} \cup \partial D_c$ is equal (up to the indeterminacy of the twisted torsion) to $\det\left(\overline{\mathcal{B}}_\rho(\beta)-t_{\mu+1}^{-1}I_{(n-1)k}\right)$. This concludes the proof when $H_1^\rho(X_{\hat{\beta}};R[H]^k)$ is torsion.

Finally, if $H_1^\rho(X_{\hat{\beta}};R[H]^k)$ is not torsion, then neither is $H_1^\rho(X_{\hat{\beta} \cup \partial D_c};R[H]^k)$ and the theorem holds trivially. 
\end{proof}

\begin{remark}
\label{rem:Torres}
The argument in the proof of Theorem \ref{thm:main} leads to an alternative proof of the twisted generalization of the Torres formula obtained by Morifuji \citep{Morifuji}.
\end{remark}

We conclude with an example of Theorem \ref{thm:main}.

\begin{example}
Assume that~$\mu=1$ and consider the braid~$\sigma_1^3 \in B_2$  whose closure is the trefoil knot~$T$. Let~$\rho: F_2 \rightarrow GL_2(\mathbb{Z}[s^{\pm 1}])$ be the representation given by 
\[
 \rho(x_1)=
\begin{pmatrix}
-s & 1 \\
0 & 1
\end{pmatrix}
, \ \
 \rho(x_2)=
\begin{pmatrix}
1 & 0 \\
s & -s
\end{pmatrix}.
\]
Using Example \ref{ex:sigma13}, we can compute the reduced twisted Burau map of the braid~$\sigma_1^3$ with respect to~$\rho$:
$$ \overline{\mathcal{B}}_\rho(\sigma_1^3)=-\rho(x_1)\rho(x_2)\rho(x_1)t^3=
\begin{pmatrix}
0 & s t^3 \\
s^2t^3 & 0
\end{pmatrix}.
$$
Consequently, one obtains 
$$ \det(\overline{\mathcal{B}}_\rho(\sigma_1^3)-I_{2})=
\det \left( \begin{pmatrix}
0 & st^3 \\
s^2t^3 & 0
\end{pmatrix}
-
\begin{pmatrix}
 1 & 0 \\
 0 & 1 \\
\end{pmatrix}
 \right)
=1-s^3t^6,
$$
and 
$$
\det(\rho(x_1) \rho(x_2)t^2-I_2)
=
\det \left( 
\begin{pmatrix}
-s & 1  \\
0 & 1 & 
\end{pmatrix}
\begin{pmatrix}
 1 & 0 \\
 s & -s \\
\end{pmatrix} t^2
-
\begin{pmatrix}
 1 & 0 \\
 0 & 1 \\
\end{pmatrix}
\right)
=1+st^2+s^2t^4.$$
The representation~$\rho$ extends to a representation of $\pi_1(S^3\setminus T)$ and Theorem \ref{thm:main} shows that (up to the indeterminacy of the twisted torsion)
$$\tau^{\rho}(\widehat{\sigma_1^3})(t)=\frac{1-s^3t^6}{1+st^2+s^2t^4}=1-st^2,$$
which coincides with the computation of Example \ref{ex:trefoil}.
\end{example}

\bibliographystyle{plain}
\nocite{*}
\bibliography{BibliographieTwisted5}

\begin{thebibliography}{10}

\bibitem{Abdulrahim}
Mohammad~N. Abdulrahim.
\newblock A faithfulness criterion for the {G}assner representation of the pure
  braid group.
\newblock {\em Proc. Amer. Math. Soc.}, 125(5):1249--1257, 1997.

\bibitem{Bigelow}
Stephen Bigelow.
\newblock The {B}urau representation is not faithful for {$n=5$}.
\newblock {\em Geom. Topol.}, 3:397--404, 1999.

\bibitem{Birman}
Joan~S. Birman.
\newblock {\em Braids, links, and mapping class groups}.
\newblock Princeton University Press, Princeton, N.J.; University of Tokyo
  Press, Tokyo, 1974.
\newblock Annals of Mathematics Studies, No. 82.

\bibitem{Boden}
Hans~U. Boden, Emily Dies, Anne~Isabel Gaudreau, Adam Gerlings, Eric Harper,
  and Andrew~J. Nicas.
\newblock Alexander invariants for virtual knots.
\newblock {\em J. Knot Theory Ramifications}, 24(3):1550009 (62 pages), 2015.

\bibitem{Burau}
Werner Burau.
\newblock \"{U}ber {Z}opfgruppen und gleichsinnig verdrillte {V}erkettungen.
\newblock {\em Abh. Math. Sem. Univ. Hamburg}, 11(1):179--186, 1935.

\bibitem{Friedl-Cha}
Jae~Choon Cha and Stefan Friedl.
\newblock Twisted torsion invariants and link concordance.
\newblock {\em Forum Math.}, 25(3):471--504, 2013.

\bibitem{Chapman}
T.~A. Chapman.
\newblock Topological invariance of {W}hitehead torsion.
\newblock {\em Amer. J. Math.}, 96:488--497, 1974.

\bibitem{C}
David Cimasoni.
\newblock A geometric construction of the {C}onway potential function.
\newblock {\em Comment. Math. Helv.}, 79(1):124--146, 2004.

\bibitem{CC}
David Cimasoni and Anthony Conway.
\newblock Colored tangles and signatures.
\newblock 2015.
\newblock arXiv:1507.07818.

\bibitem{CF}
David Cimasoni and Vincent Florens.
\newblock Generalized {S}eifert surfaces and signatures of colored links.
\newblock {\em Trans. Amer. Math. Soc.}, 360(3):1223--1264 (electronic), 2008.

\bibitem{CT}
David Cimasoni and Vladimir Turaev.
\newblock A {L}agrangian representation of tangles.
\newblock {\em Topology}, 44(4):747--767, 2005.

\bibitem{DavisKirk}
James~F. Davis and Paul Kirk.
\newblock {\em Lecture notes in algebraic topology}, volume~35 of {\em Graduate
  Studies in Mathematics}.
\newblock American Mathematical Society, Providence, RI, 2001.

\bibitem{Fox}
Ralph~H. Fox.
\newblock Free differential calculus. {I}. {D}erivation in the free group ring.
\newblock {\em Ann. of Math. (2)}, 57:547--560, 1953.

\bibitem{Friedl-Kim}
Stefan Friedl and Taehee Kim.
\newblock The {T}hurston norm, fibered manifolds and twisted {A}lexander
  polynomials.
\newblock {\em Topology}, 45(6):929--953, 2006.

\bibitem{FK2}
Stefan Friedl and Taehee Kim.
\newblock Twisted {A}lexander norms give lower bounds on the {T}hurston norm.
\newblock {\em Trans. Amer. Math. Soc.}, 360(9):4597--4618, 2008.

\bibitem{FKK}
Stefan Friedl, Taehee Kim, and Takahiro Kitayama.
\newblock Poincar\'e duality and degrees of twisted {A}lexander polynomials.
\newblock {\em Indiana Univ. Math. J.}, 61(1):147--192, 2012.

\bibitem{survey}
Stefan Friedl and Stefano Vidussi.
\newblock A survey of twisted {A}lexander polynomials.
\newblock In {\em The mathematics of knots}, volume~1 of {\em Contrib. Math.
  Comput. Sci.}, pages 45--94. Springer, Heidelberg, 2011.

\bibitem{HLN}
Jonathan~A. Hillman, Charles Livingston, and Swatee Naik.
\newblock Twisted {A}lexander polynomials of periodic knots.
\newblock {\em Algebr. Geom. Topol.}, 6:145--169 (electronic), 2006.

\bibitem{Jiang-Wang}
Bo~Ju Jiang and Shi~Cheng Wang.
\newblock Twisted topological invariants associated with representations.
\newblock In {\em Topics in knot theory ({E}rzurum, 1992)}, volume 399 of {\em
  NATO Adv. Sci. Inst. Ser. C Math. Phys. Sci.}, pages 211--227. Kluwer Acad.
  Publ., Dordrecht, 1993.

\bibitem{KL}
Paul Kirk and Charles Livingston.
\newblock Twisted {A}lexander invariants, {R}eidemeister torsion, and
  {C}asson-{G}ordon invariants.
\newblock {\em Topology}, 38(3):635--661, 1999.

\bibitem{KL2}
Paul Kirk and Charles Livingston.
\newblock Twisted knot polynomials: inversion, mutation and concordance.
\newblock {\em Topology}, 38(3):663--671, 1999.

\bibitem{KLW}
Paul Kirk, Charles Livingston, and Zhenghan Wang.
\newblock The {G}assner representation for string links.
\newblock {\em Commun. Contemp. Math.}, 3(1):87--136, 2001.

\bibitem{Kitano}
Teruaki Kitano.
\newblock Twisted {A}lexander polynomial and {R}eidemeister torsion.
\newblock {\em Pacific J. Math.}, 174(2):431--442, 1996.

\bibitem{Levine}
Jerome Levine.
\newblock Link invariants via the eta invariant.
\newblock {\em Comment. Math. Helv.}, 69(1):82--119, 1994.

\bibitem{Lin}
Xiao~Song Lin.
\newblock Representations of knot groups and twisted {A}lexander polynomials.
\newblock {\em Acta Math. Sin. (Engl. Ser.)}, 17(3):361--380, 2001.

\bibitem{LongPaton}
D.~D. Long and M.~Paton.
\newblock The {B}urau representation is not faithful for {$n\geq 6$}.
\newblock {\em Topology}, 32(2):439--447, 1993.

\bibitem{Milnor}
John Milnor.
\newblock Whitehead torsion.
\newblock {\em Bull. Amer. Math. Soc.}, 72:358--426, 1966.

\bibitem{Morifuji}
Takayuki Morifuji.
\newblock A {T}orres condition for twisted {A}lexander polynomials.
\newblock {\em Publ. Res. Inst. Math. Sci.}, 43(1):143--153, 2007.

\bibitem{Morton}
Hugh Morton.
\newblock The multivariable {A}lexander polynomial for a closed braid.
\newblock In {\em Low-dimensional topology ({F}unchal, 1998)}, volume 233 of
  {\em Contemp. Math.}, pages 167--172. Amer. Math. Soc., Providence, RI, 1999.

\bibitem{Squier}
Craig~C. Squier.
\newblock The {B}urau representation is unitary.
\newblock {\em Proc. Amer. Math. Soc.}, 90(2):199--202, 1984.

\bibitem{Swan}
Richard~G. Swan.
\newblock Projective modules over {L}aurent polynomial rings.
\newblock {\em Trans. Amer. Math. Soc.}, 237:111--120, 1978.

\bibitem{TuraevArticle}
Vladimir Turaev.
\newblock Reidemeister torsion in knot theory.
\newblock {\em Uspekhi Mat. Nauk}, 41(1(247)):97--147, 240, 1986.

\bibitem{Turaev}
Vladimir Turaev.
\newblock {\em Introduction to combinatorial torsions}.
\newblock Lectures in Mathematics ETH Z\"urich. Birkh\"auser Verlag, Basel,
  2001.
\newblock Notes taken by Felix Schlenk.

\bibitem{TuraevFaithful}
Vladimir Turaev.
\newblock Faithful linear representations of the braid groups.
\newblock {\em Ast\'erisque}, (276):389--409, 2002.
\newblock S{\'e}minaire Bourbaki, Vol. 1999/2000.

\bibitem{Wada}
Masaaki Wada.
\newblock Twisted {A}lexander polynomial for finitely presentable groups.
\newblock {\em Topology}, 33(2):241--256, 1994.

\bibitem{Weibel}
Charles~A. Weibel.
\newblock {\em An introduction to homological algebra}, volume~38 of {\em
  Cambridge Studies in Advanced Mathematics}.
\newblock Cambridge University Press, Cambridge, 1994.

\end{thebibliography}

\end{document}